\documentclass[psamsfonts]{amsart}
\usepackage[backend=biber,
sorting=none]{biblatex}  
\usepackage[utf8]{inputenc}
\usepackage{amsfonts}
\usepackage{hyperref}
\hypersetup{
pdftitle={$SO(3)$-Berezin-Toeplitz Quantization And the AJ Conjecture},
pdfsubject={Quantum Topology, Mathematical Physics, Geometric Analysis},
pdfauthor={Honghuai Fang},
pdfkeywords={Deformation Quantization, AJ Conjecture, Quantum Complexity}}
\usepackage[justification=centering]{caption}
\usepackage{amsmath}
\usepackage{xcolor}
\usepackage{amsthm}
\usepackage{pdflscape}
\usepackage{pgfplots}
\usepackage{mathrsfs}
\usepackage{amssymb}
\usepackage{float}
\usepackage{makecell}

\newtheorem{thm}{Theorem}[section]
\newtheorem{cor}[thm]{Corollary}

\newtheorem{lem}[thm]{Lemma}
\newtheorem{conj}[thm]{Conjecture}

\theoremstyle{definition}

\theoremstyle{remark}

\title{$SO(3)$-Berezin-Toeplitz Quantization And the AJ Conjecture}

\author{HONGHUAI FANG}
\email{dunjieshe@gmail.com}
\pgfplotsset{compat=1.17}
\begin{document}

\begin{abstract}
We investigate the Berezin-Toeplitz operators that operate on the geometric quantized space corresponding to the $SO(3)$-Witten-Chern-Simons theory. We conjecture that the $SO(3)$-Berezin-Toeplitz operators quantized from the A-polynomial annihilate the corresponding $SO(3)$-knot states.  
\end{abstract}

\maketitle

\section{Introduction}
In \cite{Fang2022SO3KnotSA} we investigate the geometric quantized space corresponding to the $SO(3)$-Witten-Chern-Simons theory of the torus, denoted by $\mathcal{H}^{alt}_{r+\frac{1}{2}}(j, \delta)$, which is the alternating subspace of holomorphic sections of $L^r\otimes L^{\frac{1}{2}}\otimes\delta$ over the $SU(2)$-character variety of the torus. The review is provided in Sect.2. Additionally, we put forth a conjecture that can be considered a geometric version of the volume conjecture.

\begin{conj}\label{vc}
    For any knot $K$ in $S^3$, let $Z'_r(S^3\backslash K)\in \mathcal{H}^{alt}_{r+\frac{1}{2}}(j, \delta)$ be its $r$-th $SO(3)$-knot state and let $\operatorname{vol}\left(S^{3} \backslash K\right)$ be the simplicial volume of its complement. Then we have
    \begin{equation}
        \lim _{r \rightarrow+\infty} \frac{ \pi}{r} \log ||Z'_r(S^3\backslash K)||^{2}_{r+\frac{1}{2}}=\operatorname{vol}\left(S^{3} \backslash K\right).
    \end{equation}
    
\end{conj}

 Drawing inspiration from holographic information theory, we regard the left hand side of equation \hyperref[vc]{(1)} as a form of quantum complexity that is defined on the conformal quantum space $\mathcal{H}^{alt}_{r+\frac{1}{2}}(j, \delta)$. To investigate quantum complexity, it is necessary to delve into the deformation quantization of the $SO(3)$-character variety of the torus. In this article, we focus on the study of Berezin-Toeplitz operators that operate on the geometric quantized space $\mathcal{H}^{alt}_{r+\frac{1}{2}}(j, \delta)$. Specifically, in Sec.3, we employ the Schwartz kernel of the Szego projector to demonstrate that the space of $SO(3)$-curve operators is encompassed by the space of $SO(3)$-Berezin-Toeplitz operators.

The AJ conjecture\cite{garoufalidis2004characteristic}, which establishes a connection between the A-polynomial and the colored Jones polynomial, suggests that the quantum complexity of the $SO(3)$-Berezin-Toeplitz operator quantized from the A-polynomial deserves our attention. In Sect.4, we prove the existence of a $\ast$-algebraic isomorphism between the space of $SO(3)$-Berezin-Toeplitz curve operators and the space of $SO(3)$-skein theoretical curve operators. We utilize this isomorphism to conjecture that the $SO(3)$-Berezin-Toeplitz operators quantized from the A-polynomial annihilate the corresponding $SO(3)$-knot states.

Our strategy for attacking the conjecture\hyperref[vc]{(1.1)} as follows. Let $T_{r+\frac{1}{2}}(A)$ denote the $SO(3)$-Berezin-Toeplitz operators quantized from the A-polynomial, and note that $T_{r+\frac{1}{2}}(A)Z'_r(S^3\backslash K)=0$. Our goal is to connect the quantum complexity of the operator $T_{r+\frac{1}{2}}(A)$ with the $L^2$-norm of $Z'_r(S^3\backslash K)$. There are several approaches to doing this, such as using the original definition of quantum complexity via the $\ast$-algebraic isomorphism, or studying the complexity metric on the space of $SO(3)$-Berezin-Toeplitz operators. However, we propose a novel method based on the volume conjecture and the Mahler measure of A-polynomials\cite{boyd2002mahler}. Specifically, we seek to quantize the Mahler measure and examine the semiclassical limit of the quantum Mahler measure of $T_{r+\frac{1}{2}}(A)$. Finally, we aim to establish a connection between the Mahler measure and the hyperbolic volume of the knot complement via the geometric Langlands theory. Refer to Sect.5 for further details.

\section{Review of $SO(3)$-Geometric Quantization}
In this section we review the $SO(3)$-TQFT following the geometric quantized approach\cite{Fang2022SO3KnotSA}. Let $(V,\omega)$ be a real 2-dimensional symplectic vector space equipped with a compatible linear complex structure $j$. Let $\alpha\in \Omega^1(V,\mathbb{C})$ be given by $\alpha_x(y)=\frac{1}{2}\omega(x,y)$ and endow the trivial line bundle $L=V\times \mathbb{C}$ with connection $\nabla=d-i\alpha$, the standard hermitian structure $h$, and the unique holomorphic structure, making it a prequantum line bundle. Let $\Lambda\subset V$ be a lattice of $V$ such that the symplectic volume of the fundamental domain equal to 4$\pi$. Let $L^{\frac{1}{2}}$ be the half line bundle of $L$ endowed with the standard hermitian structure $h_{\frac{1}{2}}$. Meanwhile for the canonical line bundle $K_j=\left\{\alpha \in \Omega^1(V,\mathbb{C})|\alpha(j \cdot)=i \alpha\right\}$ over $V$ we choose a half form $\delta$ of $K_j$ with an isomorphism $\varphi: \delta^{\otimes 2} \rightarrow K_j$. $K_j$ has a natural scalar product such that the square of the norm of $\alpha$ is $i \alpha \wedge \frac{\bar{\alpha} }{\omega}$. We endow $\delta$ with the hermitian structure $h_{\delta}$ making $\varphi$ an isometry.

Choose a basis $\{\mu,\lambda\}$ of $\Lambda$ such that $\omega(\mu,\lambda)=4\pi$. Let $\tau=a+bi$ be a complex number such that $\lambda=a\mu+bj\mu$. Let $p,q:V\rightarrow\mathbb{R}$ be the linear coordinate dual to $\mu,\lambda$. Thus $z=p+\tau q$ is a holomorphic coordinate of $(V,j)$. In this coordinate $\omega=4\pi dp\wedge dq$. For any $x\in \Lambda$ and any $\Phi\in\Gamma(V,L^{r}\otimes L^{\frac{1}{2}}\otimes\delta),$ define $T^{*}_x\in End(\Gamma(V,L^{r}\otimes L^{\frac{1}{2}}\otimes\delta))$ by
\begin{equation}\label{T}
   (T^{*}_x \Phi)(y)=exp\left(\frac{ (2r+1)}{4} i\omega(x, y)\right)\Phi(x+y).
\end{equation}

Denote by $\mathcal{H}_{r+\frac{1}{2}}(j,\delta)$ the $\Lambda$-invariant subspace of $H^0(V,L^r\otimes L^{\frac{1}{2}}\otimes\delta)$ and endow it with the inner product by 
$$
\left\langle\Psi_1, \Psi_2\right\rangle'_{r+\frac{1}{2}}=\int_D \left(h^{\otimes r}\otimes h_{\frac{1}{2}}\otimes h_{\delta}\right)_z\left(\Psi_1(z), \Psi_2(z)\right)\omega.
$$

There exists an orthonormal
basis $\left(\Psi_{l}\right)_{l \in \mathbb{Z} /(2r+1) \mathbb{Z}}$ of $\mathcal{H}_{r+\frac{1}{2}}(j, \delta)$ such that 
$$T^{*}_{\frac{\mu}{2r+1}}\Psi_l=exp(\frac{2l\pi i}{2r+1})\Psi_l, T^{*}_{\frac{\lambda}{2r+1}}\Psi_l=\Psi_{l+1}.$$
Specificly, $$
\Psi_0=(\frac{2r+1}{4\pi})^{\frac{1}{4}}g_0t^{r+\frac{1}{2}}\Omega_{\mu}
,\Psi_l=(T^{*}_{\frac{\lambda}{2r+1}})^l\Psi_0,$$ 
where $g_0(z)=\sum_{m\in\mathbb{Z}}exp\left(m\pi i\left((4r+2)z+(2r+1)m\tau\right)\right)$, $t(p,q)=exp(2\pi i q(p+\tau q))$, and $\Omega_{\mu}$ be a vector of $\delta$ such that $\varphi(\Omega_{\mu}^2)(\mu)=1$.

The $SU(2)$-character variety of torus $Hom(\pi_1(T^2),SU(2))/conj$ can be view as $V/(\Lambda\rtimes \mathbb{Z}_2)$. Thus we focus on the space
$$
\mathcal{H}^{alt}_{r+\frac{1}{2}}(j, \delta)\triangleq\{\Psi\in\mathcal{H}_{r+\frac{1}{2}}(j, \delta)|\Psi(z)=-\Psi(-z)\}.
$$
Therefore,  
$$\left(\Phi_l=\frac{1}{\sqrt{2}}\left(\Psi_l-\Psi_{-l}\right)\right)_{l=1,\dots,r}
$$
forms an orthonormal basis of $\mathcal{H}^{alt}_{r+\frac{1}{2}}(j, \delta)$.

Fix a knot $K$ in $S^3$ and let $\Sigma$ be the boundary of $S^3\backslash K$. Choosing an oriented diffeomorphism $\psi:\Sigma\rightarrow S^1\times S^1$ and let $\mu$ and $\lambda$ be the homology classes of $\psi^{-1}\left(S^1 \times\right.$ $\{1\})$ and $\psi^{-1}\left(\{1\} \times S^1\right)$ respectively. On the topological side, we have a hermitian space $V_r'(\Sigma)$  with an orthonormal basis $(e_n)_{n=0,\dots,r-1}$, which is the $SO(3)$-topological quantum space defined in \cite{BHMV95} through skein theory. On the geometric side, the variety $Hom(\pi_1(\Sigma),SU(2))/conj$ can be viewed as $H_1(\Sigma, \mathbb{R})/(H_1(\Sigma, \mathbb{Z})\rtimes\mathbb{Z}_2)$, where we endow $H_1(\Sigma, \mathbb{R})$ with a symplectic form $\omega(x,y)=4\pi x\cdot y$, a linear complex sturcture $j$
and a metaplectic form $\delta$. Choose $\{\mu,\lambda\}$ as a basis of $H_1(\Sigma, \mathbb{Z})$ we have a hermitian space $\mathcal{H}^{alt}_{r+\frac{1}{2}}(j, \delta)$ with an orthonormal basis $\left(\Phi_l\right)_{l=1,\dots,r}.$ For any oriented curve $\gamma\subset\Sigma$ we define an endomorphim of $\mathcal{H}^{alt}_{r+\frac{1}{2}}(j, \delta)$ by
$$T_{r+\frac{1}{2}}(\gamma)=-\left(T_{\frac{\gamma}{2r+1}}^*+T_{-\frac{\gamma}{2r+1}}^*\right).$$

In \cite{Fang2022SO3KnotSA} we prove that for any oriented curve $\gamma\subset\Sigma$, there is an isomorphism $I_r':V_r'(\Sigma)\rightarrow\mathcal{H}^{alt}_{r+\frac{1}{2}}(j, \delta)$ such that the following diagram commutes:
\begin{center}

\tikzset{every picture/.style={line width=0.75pt}} %set default line width to 0.75pt        

\begin{tikzpicture}[x=0.75pt,y=0.75pt,yscale=-1,xscale=1]
%uncomment if require: \path (0,162); %set diagram left start at 0, and has height of 162

%Straight Lines [id:da8339550763160302] 
\draw    (509,42.96) -- (587.8,41.99) ;
\draw [shift={(589.8,41.96)}, rotate = 179.29] [color={rgb, 255:red, 0; green, 0; blue, 0 }  ][line width=0.75]    (10.93,-3.29) .. controls (6.95,-1.4) and (3.31,-0.3) .. (0,0) .. controls (3.31,0.3) and (6.95,1.4) .. (10.93,3.29)   ;
%Straight Lines [id:da600013738517557] 
\draw    (482,58.96) -- (481.81,106.96) ;
\draw [shift={(481.8,108.96)}, rotate = 270.23] [color={rgb, 255:red, 0; green, 0; blue, 0 }  ][line width=0.75]    (10.93,-3.29) .. controls (6.95,-1.4) and (3.31,-0.3) .. (0,0) .. controls (3.31,0.3) and (6.95,1.4) .. (10.93,3.29)   ;
%Straight Lines [id:da5391843052990004] 
\draw    (636,56.96) -- (636.77,106.96) ;
\draw [shift={(636.8,108.96)}, rotate = 269.12] [color={rgb, 255:red, 0; green, 0; blue, 0 }  ][line width=0.75]    (10.93,-3.29) .. controls (6.95,-1.4) and (3.31,-0.3) .. (0,0) .. controls (3.31,0.3) and (6.95,1.4) .. (10.93,3.29)   ;
%Straight Lines [id:da6982064997443673] 
\draw    (515,118.96) -- (593.8,117.99) ;
\draw [shift={(595.8,117.96)}, rotate = 179.29] [color={rgb, 255:red, 0; green, 0; blue, 0 }  ][line width=0.75]    (10.93,-3.29) .. controls (6.95,-1.4) and (3.31,-0.3) .. (0,0) .. controls (3.31,0.3) and (6.95,1.4) .. (10.93,3.29)   ;

% Text Node
\draw (457,33.36) node [anchor=north west][inner sep=0.75pt]    {$V_{r} '( \Sigma )$};
% Text Node
\draw (595,29.36) node [anchor=north west][inner sep=0.75pt]    {$\mathcal{H}_{r+\frac{1}{2}}^{alt} (j,\delta )$};
% Text Node
\draw (537,20.36) node [anchor=north west][inner sep=0.75pt]    {$I_{r} '$};
% Text Node
\draw (539,95.36) node [anchor=north west][inner sep=0.75pt]    {$I_{r} '$};
% Text Node
\draw (599,106.36) node [anchor=north west][inner sep=0.75pt]    {$\mathcal{H}_{r+\frac{1}{2}}^{alt} (j,\delta )$};
% Text Node
\draw (458,109.36) node [anchor=north west][inner sep=0.75pt]    {$V_{r} '( \Sigma )$};
% Text Node
\draw (435,70.36) node [anchor=north west][inner sep=0.75pt]    {$Z_{r} '( \gamma )$};
% Text Node
\draw (643,70.36) node [anchor=north west][inner sep=0.75pt]    {$T_{r+\frac{1}{2}}( \gamma )$};
\end{tikzpicture}    
\end{center}
where $Z_{r} '( \gamma )$ is in the Kauffman skein module $K_A(T^2\times I)$ with $A=exp(\frac{2\pi i}{4r+2})$. 

\section{$SO(3)$-Berezin-Toeplitz Quantization}
In this section we show that the curve operator $\left(T_{r+\frac{1}{2}}(\gamma)\right)_{r=1,2,\dots}$ is a Berezin-Toeplitz operator. Let $L^2(V/\Lambda,L^{r}\otimes L^{\frac{1}{2}}\otimes\delta))$ be the completion of $\Gamma(V/\Lambda,L^{r}\otimes L^{\frac{1}{2}}\otimes\delta))$  with respect to the inner product $\left\langle\cdot, \cdot\right\rangle'_{r+\frac{1}{2}}$. Let $\Pi_{r+\frac{1}{2}}$ be the orthogonal projector from $L^2(V/\Lambda,L^{r}\otimes L^{\frac{1}{2}}\otimes\delta))$ to $\mathcal{H}_{r+\frac{1}{2}}(j,\delta)$. For $f\in \mathcal{C}^{\infty}(V/\Lambda)$, the \textit{Berezin-Toeplitz operator associated with $f$} is a family of operator 
$$
\left(T_{r+\frac{1}{2}}(f)=\Pi_{r+\frac{1}{2}} M_f:\mathcal{H}_{r+\frac{1}{2}}(j,\delta)\rightarrow\mathcal{H}_{r+\frac{1}{2}}(j,\delta)\right)_{r=1,2,\dots},
$$
where $M_f$ stands for the operator of multiplication by $f$. If $f$ is alternating (i.e, $f(z)=-f(-z)$), then $\left(T_{r+\frac{1}{2}}(f)\right)_{r=1,2,\dots}$ can be view as a Berezin-Toeplitz operator associated with $f$ whose elements acting on $\mathcal{H}^{alt}_{r+\frac{1}{2}}(j,\delta)$.

In addition, a \textit{Berezin-Toeplitz operator} is a family $\left(T_{r+\frac{1}{2}} \in \operatorname{End}(\mathcal{H}_{r+\frac{1}{2}}(j,\delta))\right)_{r=1,2,\dots}$ of the form
$$
T_{r+\frac{1}{2}}=\Pi_{r+\frac{1}{2}} M(f(\cdot, r))+R_r: \mathcal{H}_{r+\frac{1}{2}}(j,\delta)\rightarrow\mathcal{H}_{r+\frac{1}{2}}(j,\delta), \quad r=1,2, \ldots
$$
where $(f(\cdot, r))_r$ is a sequence of $\mathcal{C}^{\infty}(V/\Lambda)$ which admits an asymptotic expansion of the form $f_0+r^{-1} f_1+\cdots$  with coefficients $f_0, f_1, \ldots \in \mathcal{C}^{\infty}(V/\Lambda)$
and the family $\left(R_r \in \operatorname{End}(\mathcal{H}_{r+\frac{1}{2}}(j,\delta))\right)_{r}$ is a $O(r^{-\infty})$, i.e. for any $N$, there exists a positive $C_N$ such that for every $r$, $\left\|R_r\right\| \leqslant C_N r^{-N}$. We call the formal series $f_0+r^{-1} f_1+\cdots$ the total symbol of $\left(T_{r+\frac{1}{2}}\right)_r$ and $f_0$ the principal symbol of $\left(T_{r+\frac{1}{2}}\right)_r$. Also if the symbol of a Berezin-Toeplitz operator is alternating, then it can be considered as a Berezin-Toeplitz operator on $\mathcal{H}^{alt}_{r+\frac{1}{2}}(j,\delta)$

We require the following theorems, which can be readily extended to the $SO(3)$ case.

\begin{thm}\cite{Bordemann_1994}\label{t3.1}
For any $f\in \mathcal{C}^{\infty}(V/(\Lambda\rtimes\mathbb{Z}_2))$ we have that 
$$\lim _{r \rightarrow \infty}\left\|T_{r+\frac{1}{2}}(f)\right\|=\sup _{z\in V/(\Lambda\rtimes\mathbb{Z}_2)}|f(z)|.$$
\end{thm}

\begin{thm}\cite{Charles2003BerezinToeplitzOA}\label{t3.2}
    A family $\left(T_{r+\frac{1}{2}} \in \operatorname{End}(\mathcal{H}_{r+\frac{1}{2}}(j,\delta))\right)_{r=1,2,\dots}$ is a Berezin-Toeplitz operator if and only if its Schwartz kernel $\left(K_{r+\frac{1}{2}}(\cdot,\cdot)\right)_r$ satisfies the following conditions:\\
- $\left(K_{r+\frac{1}{2}}(\cdot,\cdot)\right)_r$ is a $O\left(r^{-\infty}\right)$ uniformly on any compact set of $(V/\Lambda)^2$ which does not meet the diagonal.\\
- Let $U$ be a neighborhood of diagonal. We have over $U$
$$
K_{r+\frac{1}{2}}(z, w)=\Pi_{r+\frac{1}{2}}(z,w) f(z, w, r)+O\left(r^{-\infty}\right) .
$$
where $\Pi_{r+\frac{1}{2}}(\cdot,\cdot)$ is the Schwartz kernel of the projector $\Pi_{r+\frac{1}{2}}$ and $f(\cdot, r)$ is a sequence of $\mathcal{C}^{\infty}(U)$ which admits an asymptotic expansion of the form $f_0+r^{-1} f_1+r^{-2} f_2+\cdots$.
Furthermore, the principal symbols of $\left(T_{r+\frac{1}{2}}\right)_r$ is $g_0$ which is the restriction of $f_0$ to the diagonal.
\end{thm}

Now we state the main theorem of this section:
\begin{thm}
For any oriented curve $\gamma\subset\Sigma$, $\left(T_{r+\frac{1}{2}}(\gamma)\right)_{r=1,2,\dots}$ is a Berezin-Toeplitz
operator with principal symbol $F_{\gamma}(z)=-2\cos(\frac{1}{2}\omega(\gamma,z))$.
\end{thm}

\begin{proof}
Recall that $T_{r+\frac{1}{2}}(\gamma)=-\left(T_{\frac{\gamma}{2r+1}}^*+T_{-\frac{\gamma}{2r+1}}^*\right)$. Thus it suffices to show that $\left(T_{\frac{\gamma}{2r+1}}^*\right)_r$ is a Berezin-Toeplitz operator with principal symbol $f_\gamma(z)=exp(-\frac{i}{2}\omega(\gamma,z))$. In order to using Theorem\hyperref[t3.2]{(3.2)}, we will evaluate the pointwise norm of the Schwartz kernel $\left(K_{r+\frac{1}{2}}^{(\gamma)}(\cdot,\cdot)\right)_r$ of $\left(T_{\frac{\gamma}{2r+1}}^*\right)_r$. Since $K_{r+\frac{1}{2}}^{(\gamma)}(\cdot,w)=T_{\frac{\gamma}{2r+1}}^*\Pi_{r+\frac{1}{2}}(\cdot,w)$, what we need is to compute $\Pi_{r+\frac{1}{2}}(z,w)$ and evaluate its pointwise norm, away from diagonal and near it.

Since $\Pi_{r+\frac{1}{2}}(z,w)$ is the Schwartz kernel of a projector, we have that 
$$
\Pi_{r+\frac{1}{2}}(z,w)=\sum_{l=0}^{2r}\Psi_l(z)\otimes\bar{\Psi}_l(w).
$$
From Sect.2 we know that 
$$
\Psi_l(z)=T_{\frac{\gamma}{2r+1}}^*\Psi_0(z)=\left(\frac{2r+1}{4 \pi}\right)^{\frac{1}{4}}\sum_{m\in\mathbf{Z}}exp(2\pi i \theta_{l,m}(z))t^{r+\frac{1}{2}}(z)\Omega_{\mu}(z),
$$
where $\theta_{l,m}(z)=z(l+(2r+1)m)+\frac{\tau}{4r+2}(l+(2r+1)m)^2$. Consequently, 

\begin{equation}
\begin{split}
    \Pi_{r+\frac{1}{2}}(z,w)&=\left(\frac{2r+1}{4 \pi}\right)^{\frac{1}{2}}\left(\sum_{l=0}^{2r}\sum_{m,n\in\mathbb{Z}}exp(2\pi i(\theta_{l,m}(z)-\overline{\theta_{l,n}(w)}))\right) \\
&\times t^{r+\frac{1}{2}}(z)\Omega_{\mu}(z)\otimes \bar{t}^{r+\frac{1}{2}}(w)\bar{\Omega}_{\mu}(w).
\end{split}
\end{equation}

Let $z=p_1+\tau q_1$ and $w=p_2+\tau q_2$. We evaluate $\left\|\Pi_{r+\frac{1}{2}}(z,w)\right\|$ in the following lemmas.

\begin{lem}\label{l3.4}
For every $\epsilon\in(0,\frac{1}{2}]$, if $dist(q_1-q_2,\mathbb{Z})\geq\epsilon$, then $\left\|\Pi_{r+\frac{1}{2}}(z,w)\right\|=O(r^{-\infty})$.
\end{lem}

\begin{proof}
We have that 
$$\left\|\Pi_{r+\frac{1}{2}}(z,w)\right\|\leq C\sqrt{r+\frac{1}{2}}\sum_{l=0}^{2r}\sum_{m,n\in\mathbb{Z}}exp(-2\pi\Theta_{l+(2r+1)m,l+(2r+1)m}(z,w)),$$
where
\begin{equation}
\begin{split}
    \Theta_{s,t}(z,w)&=b(r+\frac{1}{2})((q_1+\frac{s}{2r+1})^2+(q_2+\frac{t}{2r+1})^2) \\
&\geq C_1(r+\frac{1}{2})((q_1+\frac{s}{2r+1})^2+(q_1-q_2+\frac{s-t}{2r+1})^2).
\end{split}
\end{equation}

It follows that 
\begin{equation}\label{5}
\left\|\Pi_{r+\frac{1}{2}}(z,w)\right\|\leq C\sqrt{r+\frac{1}{2}}\sum_{l=0}^{2r}\sum_{m,n\in\mathbb{Z}}exp(-2\pi C_1((q_1+m+\frac{l}{2r+1})^2+(q_1-q_2+n)^2).
\end{equation}

Compare with an integral we have that 
\begin{equation}\label{6}
 \sum_{l=0}^{2r}\sum_{m\in\mathbb{Z}}exp(-(r+\frac{1}{2})(q_1+m+\frac{l}{2r+1})^2)=O(r^{\frac{3}{2}}),
\end{equation}

and using $dist(q_1-q_2,\mathbb{Z})\geq\epsilon$ we have
\begin{equation}\label{7}
\begin{split}
&\sum_{n\in\mathbb{Z}}exp(-(r+\frac{1}{2})(q_1-q_2+n)^2)\leq\sum_{n\in\mathbb{Z}}exp(-(r+\frac{1}{2})(\epsilon+n)^2)\\
&\leq 2\sum_{n=0}^{\infty}exp(-(r+\frac{1}{2})(\epsilon+n)^2)\\
&=2exp(-(r+\frac{1}{2})\epsilon^2)+\sum_{n=1}^{\infty}exp(-(r+\frac{1}{2})n^2)\\
&\leq 2exp(-(r+\frac{1}{2})\epsilon^2)+\sum_{n=1}^{\infty}exp(-(r+\frac{1}{2})n)\leq C_2 exp(-r/C_3).
\end{split}    
\end{equation}

We conclude the proof of lemma by \hyperref[5]{(5)}, \hyperref[6]{(6)}, and \hyperref[7]{(7)}.
\end{proof}

Define a section $\Pi_{r+\frac{1}{2}}^{(1)}(z,w)$ as follow:

\begin{equation}
\begin{split}
    \Pi_{r+\frac{1}{2}}^{(1)}(z,w)&=\left(\frac{2r+1}{4 \pi}\right)^{\frac{1}{2}}\left(\sum_{l=0}^{2r}\sum_{m\neq n}exp(2\pi i(\theta_{l,m}(z)-\overline{\theta_{l,n}(w)}))\right) \\
&\times t^{r+\frac{1}{2}}(z)\Omega_{\mu}(z)\otimes \bar{t}^{r+\frac{1}{2}}(w)\bar{\Omega}_{\mu}(w).
\end{split}
\end{equation}

\begin{lem}\label{l3.5}
If $|q_1-q_2|\leq\epsilon$, then $\left\|\Pi_{r+\frac{1}{2}}^{(1)}(z,w)\right\|=O(r^{-\infty})$.
\end{lem}

\begin{proof}
As in the proof of Lemma\hyperref[l3.4]{(3.4)}, it suffices to show that 
\begin{equation}
\begin{split}
\sum_{n\in\mathbb{Z}\backslash \{0\} }exp(-(r+\frac{1}{2})(q_1-q_2+n)^2)&\leq2\sum_{n=0}^{\infty}exp(-(r+\frac{1}{2})(n+\frac{1}{2})^2)\\
&\leq Cexp(-r/C_1).    
\end{split}
\end{equation}
\end{proof}

Consequently, for $(z,w)$ sufficiently close to the diagonal we have
\begin{equation}
\begin{split}
    \Pi_{r+\frac{1}{2}}(z,w)&=\left(\frac{2r+1}{4 \pi}\right)^{\frac{1}{2}}\left(\sum_{m\in\mathbb{Z}}exp(2\pi im(z-\bar{w})-\frac{2\pi b m^2}{2r+1})\right) \\
&\times t^{r+\frac{1}{2}}(z)\Omega_{\mu}(z)\otimes \bar{t}^{r+\frac{1}{2}}(w)\bar{\Omega}_{\mu}(w)+O(r^{-\infty})\\
&=\frac{2r+1}{\sqrt{8\pi b}}\left(\sum_{n\in\mathbb{Z}}exp(-\frac{(2r+1)\pi}{2b}(n-(z-\bar{w})^2))\right)\\
&\times t^{r+\frac{1}{2}}(z)\Omega_{\mu}(z)\otimes \bar{t}^{r+\frac{1}{2}}(w)\bar{\Omega}_{\mu}(w)+O(r^{-\infty}),
\end{split}
\end{equation}
where the last equality follow from Poisson's summation formula.

Define another section $\Pi_{r+\frac{1}{2}}^{(2)}(z,w)$ as
\begin{equation}
\begin{split}
 \Pi_{r+\frac{1}{2}}^{(2)}(z,w)&=\frac{2r+1}{\sqrt{8\pi b}}\left(\sum_{n\in\mathbb{Z}\backslash\{0\}}exp(-\frac{(2r+1)\pi}{2b}(n-(z-\bar{w})^2))\right)\\
&\times t^{r+\frac{1}{2}}(z)\Omega_{\mu}(z)\otimes \bar{t}^{r+\frac{1}{2}}(w)\bar{\Omega}_{\mu}(w).  
\end{split}
\end{equation}

\begin{lem}
There exists $\varepsilon>0$ such that for $|p_1-p_2|\leq\varepsilon$ and $|q_1-q_2|\leq\varepsilon$ we have $\left\|\Pi_{r+\frac{1}{2}}^{(2)}(z,w)\right\|=O(r^{-\infty})$.
\end{lem}

\begin{proof}
We have that
\begin{equation}
\left\|\Pi_{r+\frac{1}{2}}^{(2)}(z,w)\right\|\leq C(r+\frac{1}{2})\sum_{n\in\mathbb{Z}\backslash\{0\}}exp(-\frac{(2r+1)\pi}{2b}(n-(p_2-p_1)+a(q_2-q_1)^2),
\end{equation}
and we conclude as in the proof of Lemma\hyperref[l3.5]{(3.5)}
\end{proof}

Return to the proof of the theorem. Gathering all the lemmas we have 
\begin{equation}
 \begin{split}
    \Pi_{r+\frac{1}{2}}(z,w)&=\frac{2r+1}{\sqrt{8\pi b}}exp(-\frac{(2r+1)\pi}{2b}(z-\bar{w})^2)\\
&\times t^{r+\frac{1}{2}}(z)\Omega_{\mu}(z)\otimes \bar{t}^{r+\frac{1}{2}}(w)\bar{\Omega}_{\mu}(w)+O(r^{-\infty}),
\end{split}   
\end{equation}

Suppose $\gamma=p_0\mu+q_o\lambda$ and $z_0=p_o+q_0\tau$. we obtain that
\begin{equation}
K_{r+\frac{1}{2}}^{(\gamma)}(z,w)=T_{\frac{\gamma}{2r+1}}^*\Pi_{r+\frac{1}{2}}(z,w)=\Pi_{r+\frac{1}{2}}(z,w)f(z,w,r)+O(r^{-\infty}),
\end{equation}
on a small neighborhood of the diagonal, where $f(z,w,r)$ is defined as
\begin{equation}
 f(z,w,r)=exp(-\frac{\pi z_0}{b}((z-\bar{w})+\frac{z_0}{4r+2}))exp(2\pi i q_0(z+\frac{z_0}{2r+1})).   
\end{equation}

Restrict $f(z,w,r)$ to the diagonal we have 
\begin{equation}
 f(z,z,r)=exp(2\pi i(q_0p-qp_0)-\frac{\pi}{b(2r+1)}|z_0|^2) =exp(-\frac{i}{2}\omega(\gamma,z))+O(r^{-1}), 
\end{equation}
which conclude the proof by Theorem\hyperref[t3.2]{(3.2)}.
\end{proof}

\section{Toward the AJ conjecture}
In this section we will prove that as $\ast$-algebra, the space of $SO(3)$-Berezin-Toeplitz curve operators is isomorphism to the space of $SO(3)$-skein theorical curve operators. As an application, we make a conjecture that the Berezin-Toeplitz operator associated with the A-polynomial annihilates the corresponding $SO(3)$-knot state. 

Denote by $\mathcal{T}_{BT}$ the set of $SO(3)$-Berezin-Toeplitz operators and denote by $\mathcal{T}_{curv}$ the set of $SO(3)$-geometric Berezin-Toeplitz operators. We have just shown that $\mathcal{T}_{curv}\subset\mathcal{T}_{BT}$. Denote by $\mathcal{T}_{curv}^{(r)}$ the vector space $\{T_{r+\frac{1}{2}}(\gamma)|\gamma\subset\Sigma\}$ for $r=1,2,\cdots$. Define a multiplication $\ast$ on $\mathcal{T}_{curv}^{(r)}$ as 
\begin{equation}
T_{r+\frac{1}{2}}(\gamma_1)\ast T_{r+\frac{1}{2}}(\gamma_2)=T_{r+\frac{1}{2}}(\gamma_2)\circ T_{r+\frac{1}{2}}(\gamma_1),
\end{equation}
making $\mathcal{T}_{curv}^{(r)}$ into a $\ast$-algebra. 

On the topological side, denote by $\mathcal{S}_{curv}^{(r)}$ the Kauffman skein module of the cylinder over a torus $K_A(T^2\times I)$ with $A=exp(\frac{2\pi i}{4r+2})$, that is,
\begin{equation}
\mathcal{S}_{curv}^{(r)}=\{Z_{r} '( \gamma )|\gamma\subset\Sigma\}.
\end{equation}

The Kauffman skein module $\mathcal{S}_{curv}^{(r)}$ has a multiplicative structure induced by the topological operation of gluing one cylinder on top of the other. The product $\gamma_1\ast\gamma_2$ is the result of laying $\gamma_1$ over $\gamma_2$, which makes $\mathcal{S}_{curv}^{(r)}$ into an $\ast$-algebra.

\begin{thm}\label{t4.1}
There exists an isomorphism of algebras
$$
I'_r:\mathcal{T}_{curv}^{(r)}\rightarrow \mathcal{S}_{curv}^{(r)}.
$$
\end{thm}

\begin{proof}
 Denote by $\mathcal{QT}^{(r)}$ the quantum torus with deformation parameter $A=exp(\frac{2\pi i}{4r+2})$, that is,
$$
\mathcal{QT}^{(r)}=\mathbb{C}[A^{\pm 1}]<m^{\pm 1},l^{\pm 1}|lm=A^2ml>.
$$

Consider the basis over $\mathbb{C}$ of $\mathcal{QT}^{(r)}$ given by the vector $e_{a,b}=A^{ab}m^a l^{-b}$. Define the multiplication $\ast$ by
\begin{equation}
e_{a,b}\ast e_{c,d}=A^{ad-bc}e_{a+b,c+d}.
\end{equation}
 Thus $\mathcal{QT}^{(r)}$ is a noncommutative algebra. Consider the algebra morphism
 $$
\sigma:\mathcal{QT}^{(r)}\rightarrow\mathcal{QT}^{(r)},\sigma(e_{a,b})=e_{-a,-b},
 $$
 and let $\mathcal{QT}^{(r)}_{\sigma}$ be its invariant part, which is spanned by $e_{a,b}+e_{-a,-b}$. It has been shown in \cite{Frohman1998SkeinMA} that there exsit an isomorphism of $\ast$-algebras between $\mathcal{QT}^{(r)}_{\sigma}$ and 
$\mathcal{S}_{curv}^{(r)}$. We show that $\mathcal{T}_{curv}^{(r)}$ is isomorphic to $\mathcal{QT}^{(r)}_{\sigma}$ as $\ast$-algebras.

 Let $\lambda_{a,b}=a\mu+b\lambda$, $a,b\in\mathbb{Z}$. For any $\Phi\in\mathcal{H}^{alt}_{r+\frac{1}{2}}(j, \delta)$, we compute that

 \begin{equation}
 \begin{split}
      T^{*}_{\frac{\lambda_{a,b}}{2r+1}}\Phi(p,q)&=exp(\frac{i}{4}\omega(a\mu+b\lambda,p+q\tau))\Phi(p+\frac{a}{2r+1},q+\frac{b}{2r+1})\\
      &=exp(i\pi(aq-bp))\Phi(p+\frac{a}{2r+1},q+\frac{b}{2r+1}),
 \end{split}
 \end{equation}

 \begin{equation}
 \begin{split}
      T^{*}_{\frac{\lambda_{c,d}}{2r+1}}\circ T^{*}_{\frac{\lambda_{a,b}}{2r+1}}\Phi(p,q)&=exp(i\pi(a(q+\frac{d}{2r+1})-b(p+\frac{c}{2r+1}))\\
      &\times exp(\frac{i}{4}\omega(c\mu+d\lambda,p+q\tau))\Phi(p+\frac{a+c}{2r+1},q+\frac{b+d}{2r+1})\\
      &=exp(\frac{\pi i(ad-bc)}{2r+1})exp(\pi i((a+c)q-(b+d)p))\Phi(p+\frac{a+c}{2r+1},q+\frac{b+d}{2r+1})\\
      &=A^{ad-bc} T^{*}_{\frac{\lambda_{a+c,b+d}}{2r+1}}\Phi(p,q).
 \end{split}
 \end{equation}

 It follows that 
 \begin{equation}\label{22}
 \begin{split}
T_{r+\frac{1}{2}}(\lambda_{a,b})\ast T_{r+\frac{1}{2}}(\lambda_{c,d})&=( T^{*}_{\frac{\lambda_{c,d}}{2r+1}}+ T^{*}_{\frac{-\lambda_{c,d}}{2r+1}})\circ( T^{*}_{\frac{\lambda_{a,b}}{2r+1}}+ T^{*}_{\frac{-\lambda_{a,b}}{2r+1}})\\
&=A^{ad-bc} (T^{*}_{\frac{\lambda_{a+c,b+d}}{2r+1}}+T^{*}_{\frac{-\lambda_{a+c,b+d}}{2r+1}})\\
&+A^{bc-ad} (T^{*}_{\frac{\lambda_{a-c,b-d}}{2r+1}}+T^{*}_{\frac{-\lambda_{a-c,b-d}}{2r+1}})\\
&=-A^{ad-bc}T_{r+\frac{1}{2}}(\lambda_{a+c,b+d})-A^{bc-ad}T_{r+\frac{1}{2}}(\lambda_{a-c,b-d}).
 \end{split}
 \end{equation}

 Also we have 
 \begin{equation}\label{23}
 (e_{a,b}+e_{-a,-b})\ast(e_{c,d}+e_{-c,-d})=A^{ad-bc}e_{a+c,b+d}+A^{bc-ad}e_{a-c,b-d}.    
 \end{equation}

 From \hyperref[22]{(22)} and \hyperref[23]{(23)} we conclude that there exists an isomorphism of $\ast$-algebras from $\mathcal{T}_{curv}^{(r)}$ to $\mathcal{QT}^{(r)}_{\sigma}$ such that the image of $T_{r+\frac{1}{2}}(\lambda_{a,b})$ is $-(e_{a,b}+e_{-a,-b})$.
\end{proof}

\begin{cor}\label{c4.2}
The space of $SO(3)$-Berezin-Toeplitz curve operators $\mathcal{T}_{curv}^{(r)}$ is a finitely generated $\ast$-algebra. 
\end{cor}
\begin{proof}
It has been shown in \cite{bullock1999multiplicative} that $\mathcal{S}_{curv}^{(r)}$ is generated over $\mathbb{C}[A^{\pm 1}]$ by $Z_r^{'}(\lambda_{1,0})$, $Z_r^{'}(\lambda_{0,1})$ and $Z_r^{'}(\lambda_{1,1})$. 
\end{proof}

Let $Y$ be a manifold whose fundamental group is finitely generated. Denote by $\chi(Y)$ and $\mathcal{M}(Y)$ its $SL(2,\mathbb{C})$-character variety and $SU(2)$-character variety respectively. Every pair of generators $\mu,\lambda$ of $\pi_1(T^2)$ will define an isomorphism between $\chi(T^2)$ and $(\mathbb{C}^{*})^2/\sigma$, where $(\mathbb{C}^{*})^2$ is the set of non-zero complex pairs $(m,l)$ and $\sigma$ is the involution $\sigma(m,l)=(m^{-1},l^{-1})$. Given a knot $K$ in $S^3$, let $X$ be its complement. An orientation of $K$ will define a unique pair of an oriented meridian and an oriented longitude such that the linking number between the longitude and the knot is zero. The pair provides an identification of $\chi(\partial X)$ and $(\mathbb{C}^{*})^2/\sigma$ which does not depend on the orientation of $K$.

The inclusion map $T^2=\partial X\hookrightarrow X$ induces a map $\theta:\chi(X)\rightarrow \chi(\partial X)$. The Zariski closure of the lift of $\theta(\chi(X))$ under the projection $(\mathbb{C}^{*})^2\rightarrow (\mathbb{C}^{*})^2/\sigma$ is an algeraic set in $\mathbb{C}^2$ consisting of components of dimension 0 or 1. The union of all the one-dimension components is defined by a polynomial $A_K(m,l)\in\mathbb{Z}[m,l]$ whose coefficients are co-prime, which is by definition the \textit{A-polynomial} of $K$. 

Note that $\mathcal{M}(\partial X)$ can be embedded into $\chi(\partial X)$ by sending $(p,q)$ to $(e^{-2\pi i q},e^{-2\pi i p})$. Thus $A(p,q):=-\left(A_K(e^{-2\pi i q},e^{-2\pi i p})+A_K(e^{2\pi i q},e^{2\pi i p})\right)$ can be viewed as a smooth function in $\mathcal{C}^{\infty
}(\mathcal{M}(\partial X))$. We want to show that the $SO(3)$-Berezin-Toeplitz operator $T_{r+\frac{1}{2}}(A)$ annihilates the corresponding $SO(3)$-knot state $Z'_r(S^3\backslash K)$, which can be viewed as a geometric version of the AJ conjecture.

The restriction map $\theta:\chi(X)\rightarrow \chi(\partial X)$ induces a map $\hat{\theta}:\mathbb{C}[\chi(\partial X)]\rightarrow \mathbb{C}[\chi(X)]$ between rings of regular functions. Denote by $\mathcal{I}(K)$ the kernel of $\hat{r}$. Since the ring $\mathbb{C}[\chi(\partial X)]$ can be embedded naturally into the principal ideal domain $\mathbb{C}(m)[l^{\pm{1}}]$, we can consider the ideal extension of $\mathcal{I}(K)$ which is generated by a single polynomial that is $m$-equivalent with the A-polynomial. 

In the quantum case, there is a map $\hat{\Theta}_r:\mathcal{S}_{curv}^{(r)}\rightarrow K_{e^{\frac{\pi i}{2r+1}}}(X)$ obtained by gluing the cylinder over a torus into $X$ at the $T^2\times \{0\}$ so that the meridian goes to the meridian and the longitude goes to the longitude. Denote by $\mathcal{I}_r(K)$ the kernel of $\hat{\Theta}_r$. We need the following theorem proved in \cite{Frohman1998TheAF}.

\begin{thm}\label{t4.3}
$\lim\limits_{r \rightarrow+\infty}\mathcal{I}_r(K)=\mathcal{I}(K)$; Furthermore, let $\mathcal{F}_r(K)$ be the annihilator of $Z'_r(S^3\backslash K)$ in $\mathcal{S}_{curv}^{(r)}$. Then the ideal $\mathcal{I}_r(K)$ lies in $\mathcal{F}_r(K)$.
\end{thm}

Now Suppose the ideal extension of $\mathcal{I}_r(K)$ is generated by a single polynomial $\alpha_K^{(r)}(m,l,q=e^{\frac{2\pi i}{2r+1}})$ such that 
$$
\lim\limits_{r \rightarrow+\infty}\alpha_K^{(r)}(m,l,q=e^{\frac{2\pi i}{2r+1}})=-(A_K(m,l)+A_K(m^{-1},l^{-1})).
$$ 

By Theorem\hyperref[t4.1]{(4.1)} we obtain that $\left((I^{'}_r)^{-1}(\alpha_K^{(r)})\right)_r$ is a $SO(3)$-Berezin-Toeplitz curve operator. Suppose $(I^{'}_r)^{-1}(\alpha_K^{(r)})=T_{r+\frac{1}{2}}(F_{\gamma})$, where $F_{\gamma}\in\mathbb{Z}[e^{2\pi i(aq+bp)}]$ . From some symbol calculation one could show that 
\begin{equation}
\lim\limits_{r\rightarrow\infty}\|T_{r+\frac{1}{2}}(A-F_{\gamma})\|=0.    
\end{equation}

By Theorem\hyperref[t3.1]{(3.1)} we have that $A=F_{\gamma}$. By Theorem\hyperref[t4.3]{(4.3)} it follows that $T_{r+\frac{1}{2}}(A)$ should lie in the annihilator of $Z'_r(S^3\backslash K)$.
\begin{conj}\label{c4.4}
 The $SO(3)$-Berezin-Toeplitz operator associated with $A(p,q)$ annihilates the corresponding $SO(3)$-knot state, that is 
 \begin{equation}
     T_{r+\frac{1}{2}}(A)Z'_r(S^3\backslash K)=0,r=1,2,\cdots.
 \end{equation}
\end{conj}

\section{Discussion}
Drawing inspiration from the geometric volume conjecture\hyperref[vc]{(1.1)}, the geometric AJ conjecture\hyperref[c4.4]{(4.4)}, and the holographic complexity=volume conjecture\cite{brown2016complexity}, it is anticipated that the quantum complexity associated with the $SO(3)$-Berezin-Toeplitz curve operator $T_{r+\frac{1}{2}}(A)$ captures relevant information pertaining to the $L^2$-norm of the $SO(3)$-knot state $Z'_r(S^3\backslash K)$ and the simplicial volume of $S^3\backslash K$. To investigate quantum complexity on the space $\mathcal{T}_{curv}^{(r)}$, one could employ the circuit complexity approach, as demonstrated in Corollary\hyperref[c4.2]{(4.2)}, or alternatively, a complexity metric could be introduced on $\mathcal{T}_{BT}^{(r)}$ using Nielsen's method.

Additionally, there exists a quantity that relates A-polynomials with hyperbolic volumes of knot complements. This quantity is known as the Mahler measure of an A-polynomial, denoted by $m(A_K)\in\mathbb{C}$, which has connections to L-functions and dilogarithm functions. It is conjectured that the Mahler measure is essentially equal to the volume of $S^3\backslash K$. We anticipate that the quantum counterpart of the Mahler measure, denoted by $\hat{m}$, is a type of quantum complexity that has the property that the semiclassical limit of $\hat{m}(T_{r+\frac{1}{2}}(A))$ equals $m(A)$. Moreover, the quantum modular property of the colored Jones polynomial suggests that one can link the Mahler measure with volume through the geometric Langlands theory. Further details will be discussed in upcoming work.

\begin{table}[!h]
        \centering
\resizebox{.8\columnwidth}{!}{        
\begin{tabular}{|p{0.25\textwidth}|p{0.25\textwidth}|p{0.25\textwidth}|p{0.25\textwidth}|}
\hline 
  & \begin{center}
{\Large Quantum}
\end{center}
 & \begin{center}
{\Large Galois}
\end{center}
 & \begin{center}
{\Large Automorphic}
\end{center}
 \\
\hline 
 \begin{center}
$\displaystyle T^{2}$
\end{center}
 & \begin{center}
$\displaystyle \mathcal{H}_{r+1/2}^{alt}( j,\delta )$
\end{center}
 & \begin{center}
$\displaystyle \chi \left( T^{2}\right)$
\end{center}
 & \begin{center}
$\displaystyle Bun_{^{L} G}\left( T^{2}\right)$
\end{center}
 \\
\hline 
 \begin{center}
$\displaystyle S^{3} \backslash K$
\end{center}
 & \begin{center}
$\displaystyle Z_{r}^{'}\left( S^{3} \backslash K\right)$
\end{center}
 & \begin{center}
$\displaystyle \chi \left( S^{3} \backslash K\right)$
\end{center}
 & \begin{center}
$\displaystyle Bun_{^{L} G}\left( S^{3} \backslash K\right)$
\end{center}
 \\
\hline 
 \begin{center}
Complexity
\end{center}
 & \begin{center}
$\displaystyle \frac{log\| Z_{r}^{'}\left( S^{3} \backslash K\right) \| }{r}$
\end{center}
 & \begin{center}
$\displaystyle m( A_{K})$ 
\end{center}
 & \begin{center}
$\displaystyle \frac{1}{\pi } vol\left( S^{3} \backslash K\right)$
\end{center}
 \\
 \hline
\end{tabular}}
        
        \end{table}

\printbibliography

@article{BHMV95,
  title={Topological quantum field theories derived from the Kauffman bracket},
  author={Blanchet, Christian and Habegger, Nathan and Masbaum, Gregor and Vogel, Pierre},
  journal={Topology},
  volume={34},
  number={4},
  pages={883--928},
  year={1995},
  publisher={Oxford; New York: Pergamon Press, 1962-}
}

@article{brown2016complexity,
  title={Complexity, action, and black holes},
  author={Brown, Adam R and Roberts, Daniel A and Susskind, Leonard and Swingle, Brian and Zhao, Ying},
  journal={Physical Review D},
  volume={93},
  number={8},
  pages={086006},
  year={2016},
  publisher={APS}
}

@article{garoufalidis2004characteristic,
  title={On the characteristic and deformation varieties of a knot},
  author={Garoufalidis, Stavros},
  journal={Geom. Topol. Monogr},
  volume={7},
  pages={291--304},
  year={2004}
}

@article{Bordemann_1994,
  
	year = 1994,
	month = {oct},
  
	publisher = {Springer Science and Business Media {LLC}
},
  
	volume = {165},
  
	number = {2},
  
	pages = {281--296},
  
	author = {Martin Bordemann and Eckhard Meinrenken and Martin Schlichenmaier},
  
	title = {Toeplitz quantization of Kähler manifolds anggl(N), N$\rightarrow$$\infty$ limits},
  
	journal = {Communications in Mathematical Physics}
}

@article{Charles2003BerezinToeplitzOA,
  title={Berezin-Toeplitz Operators, a Semi-Classical Approach},
  author={Laurent Charles},
  journal={Communications in Mathematical Physics},
  year={2003},
  volume={239},
  pages={1-28}
}

@article{Frohman1998TheAF,
  title={The A-polynomial from the noncommutative viewpoint},
  author={Charles Frohman and Rǎzvan Gelca and Walter Lofaro},
  journal={Transactions of the American Mathematical Society},
  year={1998},
  volume={354},
  pages={735-747}
}

@article{Frohman1998SkeinMA,
  title={Skein modules and the noncommutative torus},
  author={Charles Frohman and Rǎzvan Gelca},
  journal={Transactions of the American Mathematical Society},
  year={1998},
  volume={352},
  pages={4877-4888}
}

@inproceedings{Fang2022SO3KnotSA,
  title={SO(3)-Knot States and the Volume Conjecture},
  author={Honghuai Fang},
  year={2022}
}

@article{boyd2002mahler,
  title={Mahler’s measure and invariants of hyperbolic manifolds},
  author={Boyd, David W},
  journal={Number theory for the millennium, I (Urbana, IL, 2000)},
  volume={127},
  pages={143},
  year={2002}
}

@misc{bullock1999multiplicative,
      title={Multiplicative structure of Kauffman bracket skein module quantizations}, 
      author={Doug Bullock and Jozef H. Przytycki},
      year={1999},
      archivePrefix={arXiv},
      primaryClass={math.QA}
}

\end{document}